\documentclass{amsart}

\input xy
\xyoption{all}

\usepackage{soul}

\newbox\noforkbox \newdimen\forklinewidth
\setbox0\hbox{$\textstyle\smile$}
\setbox1\hbox to \wd0{\hfil\vrule width \forklinewidth depth-2pt
 height 10pt \hfil}
\setbox\noforkbox\hbox{\lower 2pt\box1\lower 2pt\box0\relax}
\def\unionstick{\mathop{\copy\noforkbox}\limits}
\def\nonfork_#1{\unionstick_{\textstyle #1}}

\setbox0\hbox{$\textstyle\smile$}
\setbox1\hbox to \wd0{\hfil{\sl /\/}\hfil}
\setbox2\hbox to \wd0{\hfil\vrule height 10pt depth -2pt width
              \forklinewidth\hfil}
\newbox\doesforkbox
\setbox\doesforkbox\hbox{\lower 2pt\box1 \lower 2pt\box2\lower2pt\box0\relax}
\def\nunionstick{\mathop{\copy\doesforkbox}\limits}

\def\fork_#1{\nunionstick_{\textstyle #1}}

\newbox\noforkboxs \newdimen\forklinewidths
\setbox0\hbox{$\textstyle\smile$}
\setbox1\hbox to \wd0{\hfil\vrule width \forklinewidths depth-2pt
 height 10pt \hfil}
\setbox\noforkboxs\hbox{\lower 2pt\box1\lower 2pt\box0\relax}
\def\unionsticks{\mathop{\copy\noforkboxs^{\mkern-8mu\mathrm{s}}}\limits}
\def\nonforks_#1{\unionsticks_{\textstyle #1}}

\setbox0\hbox{$\textstyle\smile$}
\setbox1\hbox to \wd0{\hfil{\sl /\/}\hfil}
\setbox2\hbox to \wd0{\hfil\vrule height 10pt depth -2pt width
              \forklinewidth\hfil}
\newbox\doesforkboxs
\setbox\doesforkboxs\hbox{\lower 2pt\box1 \lower 2pt\box2\lower2pt\box0\relax}
\def\nunionsticks{\mathop{\copy\doesforkbox}\limits}

\def\forks_#1{\nunionsticks_{\textstyle #1}}

\usepackage{amsmath}
\usepackage{amssymb}

\usepackage{amsthm}  

\usepackage{comment}

\newtheorem{same}{This should never appear}[section]
\newtheorem{defin}[same]{Definition}

\newtheorem{remark}[same]{Remark}
\newtheorem{theorem}[same]{Theorem}

\newtheorem{lemma}[same]{Lemma}
\newtheorem{fact}[same]{Fact}
\newtheorem{question}[same]{Question}
\newtheorem{cor}[same]{Corollary}
\newtheorem{prop}[same]{Proposition}

\newtheorem*{mtheorem}{Theorem \ref{Thm:main_Kl}}
\newtheorem*{mtheorem2}{Theorem \ref{Cor:Kk_not_kappa_tame}}
\newcommand{\Cl}{\operatorname{cl}}

\newtheorem{defin*}{defin}
\newtheorem*{theorem*}{Theorem}

\makeatletter
\newcommand{\skipitems}[1]{%
  \addtocounter{\@enumctr}{#1}%
}
\newcommand{\K}{\mathbf{K}}

\newcommand{\id}{\operatorname{id}}

\newcommand{\Kk}{\mathbf{K}_2}
\newcommand{\Kl}{\mathbf{K}_1}

\newcommand{\restr}{\upharpoonright}

\newcommand{\LS}{\operatorname{LS}}

\newcommand{\Z}{\mathbb{Z}}

\newcommand{\gtp}{\mathbf{gtp}}
\newcommand{\gS}{\mathbf{gS}}

\usepackage{tikz-cd}
\newcommand{\Q}{\mathbb{Q}}

\title{Examples of non-tame abstract elementary classes of abelian groups}

\author[Herden, Mazari-Armida, Walton]{Daniel Herden, Marcos Mazari-Armida, Michael D. Walton}

\thanks{The first author was supported by Simons Foundation grant MPS-TSM-00007788.
	The second author was supported by NSF grant DMS-2348881 and Simons Foundation grant MPS-TSM-00007597}

\address{\newline Daniel Herden, Marcos Mazari-Armida, and Michael D. Walton \newline Department of Mathematics \newline Baylor University \newline Sid Richardson Building \newline 1410 S.~4th Street \newline	Waco, TX 76706, USA}
\urladdr{https://sites.baylor.edu/daniel\_herden/}
\urladdr{https://sites.baylor.edu/marcos\_mazari/}

\email{daniel\_herden@baylor.edu; marcos\_mazari@baylor.edu, michael\_walton2@baylor.edu}

\begin{document}

\begin{abstract} 
We construct an abstract elementary class $\Kl$ of torsion-free abelian groups such that $\Kl$ is not $(<\aleph_0)$-tame but is $\aleph_0$-tame. This  answers a question of \cite{bova-tame}. Furthermore, for every regular uncountable cardinal $\mu$ less than the first measurable cardinal, we construct an abstract elementary class $\Kk(2^\mu)$ of torsion-free abelian groups such that $\Kk(2^\mu)$ is not $(<\mu)$-tame.

$\Kl$ and $\Kk(2^\mu)$ are non-tame for algebraic reasons. Furthermore, they constitute the first examples of non-tame abstract elementary classes in a natural language.
\end{abstract}


\maketitle

{\let\thefootnote\relax\footnote{{AMS 2020 Subject Classification:
Primary:  03C60, 03C48. Secondary: 03C45, 20K20,  	20K40.

  }}}


\section{Introduction}

Abstract elementary classes (AECs for short) \cite{sh88} are the prominent setup for studying non-elementary classes of structures as they are general enough to encompass many interesting examples while at the same time allowing the development of a deep theory.  For examples of AECs see \cite[Section 3.2]{bova-tame}, \cite[Example 2.2]{maztor}, \cite[Section 3.2]{bon25}; and for the development of the theory see \cite{grossberg2002}, \cite{shelahaecbook}, \cite{baldwinbook09}, \cite{bova-tame}.

An abstract elementary class is an object determined by a class of structures $K$ and a partial order $\leq_K$ on  $K$ which satisfies a few natural closure properties. Two of the key properties are that the class is closed under increasing chains and  that every subset of an element in $K$ can be closed to a \emph{small} element in $K$ (see Definition~\ref{Def:AEC}). In this paper, $K$ will always be a class of torsion-free abelian groups and $\leq_K$ will strengthen the subgroup relation.

As an AEC is a semantic object, the correct notion of type in this setup is not a collection of formulas as in first-order model theory, but that of a Galois type \cite{sh300}. Intuitively, the Galois type of an element $b$ in $N\in K$ over the set of parameters~$A$ (for $A \subseteq N$) is the equivalence class of all triples $(b',A,N')$ such that there
is a way of sending $b$ to $b'$ while simultaneously keeping all parameters in $A$ fixed (see Definition~\ref{Def:Galois_type}).

Since Galois types are semantic objects, it might not be possible to distinguish two types by a \emph{small} subset of their set of parameters. When we can distinguish types by subsets of size less than $\lambda$, we say that an AEC is $(<\lambda)$-\emph{tame}. It is worth pointing out that elementary classes are $(<\aleph_0)$-tame\footnote{Where $\aleph_0$ is the first infinite cardinal.} as two first-order types can be distinguished by a formula which has a finite parameter.

Tameness was isolated by Grossberg and VanDieren \cite{tamenessone} twenty five years ago and since then has become one of the central notions in the study of AECs. Nowadays, many results on AECs either outright assume tameness, see \cite{tamenesstwo}, \cite{bkv}, \cite{extendingframes}, \cite{grva} \cite{vaseyt}, \cite{leu2}, \cite{may}, or showing tameness (often with  assumptions beyond ZFC) is an intermediate step toward a larger goal, see \cite{vaseyu}, \cite{vaseye}, \cite{shva}.

There are plenty of tame AECs. For instance, all the following AECs with pure embeddings are $(<\aleph_0)$-tame: torsion abelian groups, $\aleph_1$-free abelian groups, and absolutely pure $R$-modules. Many more examples of tame AECs, both of algebraic and non-algebraic flavor, are provided in \cite[Section 3.2]{bova-tame}. Moreover, the statement that every AEC is $(<\lambda)$-tame for some cardinal $\lambda$ is equivalent to a large cardinal axiom \cite{bo-un}, \cite{bo-li}.

On the other hand, there are only three families of examples known to be non-tame: Hart-Shelah, see \cite{ha-sh}, \cite{ba-ko}, Baldwin-Shelah \cite{bash}, and Boney-Unger-Shelah, see \cite{bo-un}, \cite{sh-ntame}.   The Hart-Shelah and Boney-Unger-Shelah examples are of set-theoretic nature, while the Baldwin-Shelah example is at the border of algebra and set theory. Further details about these three examples can be consulted in \cite[Section 3.2.2]{bova-tame}.

In this paper, we provide a new family of examples of non-tame abstract elementary classes. The reason tameness fails in our examples is a purely algebraic one. Intuitively, tameness fails in all of our classes since for a small cardinal $\lambda$\footnote{We need $\lambda$ to be a regular cardinal below the first measurable cardinal.} every group homomorphism $f: \prod_{\alpha<\lambda} \Z e_\alpha \to \prod_{\alpha<\lambda} \Z e_\alpha$ is completely determined by the values $f(e_\alpha)$ (see Theorem \ref{Prop:Homo_Unique} and Theorem~\ref{Prop:measurable_homo}).  We leverage this result and do some minor coding via divisibility of primes to obtain two distinct Galois types which are equal when restricted to small subsets (see Theorem \ref{Thm:Kl_not_aleph0_tame} and Theorem~\ref{Cor:Kk_not_kappa_tame}).

In stark contrast to previous examples of non-tameness, set-theoretic notions play a minimal role in our examples. It is worth emphasizing that in the case of $(<\aleph_0)$-tameness, set-theoretic notions play no role at all. Another significant difference between our examples and the earlier ones is that our examples are in a natural language, the language of abelian groups\footnote{The language of abelian groups is $L= \{ +, -, 0\}$  where $+,-$ are binary functions.}, while all previous examples are considerably more pathological, relying on big languages to encode non-tameness.

We first construct an AEC of torsion-free abelian groups $\Kl$ (see Definition \ref{Def:Kl(p_1,p_2)}) with the following properties:

\begin{mtheorem}
 $\Kl$ is an abstract elementary class with $\LS(\Kl)=\aleph_0$ such that:
    \begin{enumerate}
        \item  $\Kl$ is \emph{not} $(<\aleph_0)$-tame.
        
        \item  $\Kl$ is $\aleph_0$-tame.
            
        \item  $\Kl$ is stable in $\lambda$ if and only if $\lambda^{\aleph_0} = \lambda$. 
    \end{enumerate}
\end{mtheorem}

The existence of such an AEC answers positively a question of \cite[p.~412]{bova-tame}\footnote{More precisely it answers: \emph{Is there an AEC $\K$ that is $\aleph_0$-tame but not $(< \aleph_0)$-tame?}}. An important model-theoretic difference between this example and previous examples of failure of tameness  is that $\Kl$ is stable on unboundably many cardinals.

Furthermore, in  \cite[p. 413]{bova-tame} it is asked: \emph{Are there
natural mathematical structures that are, in some sense, well behaved and should be amenable to a model-theoretic analysis,
but are not tame?} We argue that $\Kl$ provides the first evidence that the answer to this question might be positive as $\Kl$ is stable and is not $(<\aleph_0)$-tame. Nevertheless, we still think that significant work needs to be done to provide a final answer.

Using the same core idea as that for $\Kl$, we construct for every regular uncountable cardinal $\mu$ less than the first measurable cardinal an AEC of torsion-free abelian groups $\Kk(2^\mu)$ (see Definition \ref{Def:Kk(p_1,p_2)}) such that the following holds:

\begin{mtheorem2} Assume $\mu$ is a regular  uncountable cardinal less than the first measurable cardinal.
Then $\Kk(2^\mu)$ is an abstract elementary class with $\LS(\Kk(2^\mu))=2^\mu$ such that $\Kk(2^\mu)$ is \emph{not} $(<\mu)$-tame.
\end{mtheorem2}

The AECs $\Kl$ and $\Kk(2^\mu)$  are similar, but the failure of tameness in $\Kl$ is more localized as it only happens in  ``small'' models (see Lemma \ref{Lem:Kl_aleph_theta_tame}) while the failure of tameness in $\Kk(2^\mu)$ happens in models of arbitrary size (see Lemma \ref{Lem:Kk_not_mu_theta_tame} and Lemma \ref{kk-no-aleph-0-tame}).

The paper is divided into three sections. Section \ref{Section:Prelims} covers the notions of abelian group theory, abstract elementary classes, and the AEC of torsion-free abelian groups with pure embeddings which we use in this paper. Section \ref{sectio-3} covers the construction of $\Kl$ and its analysis. Section \ref{Section:Kk} covers the construction of $\Kk(\kappa)$ and its analysis.

\section{Preliminaries}\label{Section:Prelims}

We introduce the notions of abelian group theory and abstract elementary classes that will be used in this paper. Further details on abelian group theory can be consulted in \cite{fuc} and on abstract elementary classes in \cite{baldwinbook09}. 


\subsection*{Abelian groups}

All groups discussed in this paper are abelian groups.

An abelian group $G$ is \emph{torsion-free} if every $0\neq g \in G$ has infinite order. A subgroup $G$ of $H$ is \emph{pure}, denoted by $G \leq_p H$, if for every $n\in \Z$, $nG = G\cap nH$, where $nG = \{ng : g\in G\}$.  When $G, H$ are torsion-free groups, if $pG = G\cap pH$ for every prime number $p$, then $G \leq_p H$.  This follows from the fact that for torsion-free abelian groups if $G\leq_p H$ and $nh\in G$ for $h\in H$ and $0\ne n\in \Z$, then $h\in G$.

Given a group $G$ and $p$ a prime number, let $p^\omega G := \bigcap_{n\in \Z_{\geq 0}} p^n G = \{ g \in G : p^n \mid g \text{ for all } n < \omega \}$. 

\begin{fact}\label{Facts:Uncited} \ 
    Let $p$ be a prime and $G, H$  be torsion-free abelian groups.
    \begin{enumerate}
    \item If $G \leq_p H$, then $p^\omega G = p^\omega H \cap G$. 
        \item $p^\omega G \leq_p G$ and $p^\omega(p^\omega G)= p^\omega G$.
        \item If $\{G_i : i<\lambda\}$ is such that $G_i \leq_p H$ for every $i<\lambda$, then $p^\omega \left(\bigcap_{i<\lambda} G_i\right) = \bigcap_{i<\lambda} \left(p^\omega G_i\right)$. 
    \end{enumerate}
\end{fact}

Finally, given primes $p$ and $q$, let $\Z[1/p, 1/q]$ denote the subring of $\Q$ that results from adjoining the rational numbers $1/p, 1/q$ to $\Z$, i.e.,  \[\Z[1/p, 1/q] = \{ p^{-n}a + q^{-m}b : n, m \in \Z_{\geq 0}, a, b \in \Z\}.\]

\subsection*{Abstract elementary classes}

Abstract elementary classes (or AECs) were introduced by Shelah in the seventies \cite{sh88}. Given a structure $M$, we will write $|M|$ for the underlying set of the structure and $\|M\|$ for its cardinality.

\begin{defin}\label{Def:AEC}
    An \emph{abstract elementary class} is a pair $\K = (K, \leq_{\K})$ where:
    \begin{enumerate}
        \item $K$ is a class of $L$-structures (for some fixed finitary language $L = L(\K)$).
        \item $\leq_{\K}$ is a partial order on $K$.
        \item If $M \leq_\K N$, then $M$ is a substructure of $N$.\footnote{In this paper, the substructure  relation will be the subgroup relation $\leq$ of abelian groups.}
        \item If $M \leq_\K N$ are in $K$ and $f: N \cong N'$ is an isomorphism, then $f[M] \leq_\K N'$ for the image $f[M]$ of $f$. In particular, $K$ is closed under isomorphic images.
        \item Coherence: If $M_0, M_1, M_2 \in K$ satisfy $M_0 \leq_\K M_2$, $M_1 \leq_\K M_2$, and $M_0 \subseteq M_1$, then $M_0 \leq_\K M_1$.
        \item Tarski-Vaught axioms: Suppose $\delta$ is a limit ordinal and $\{M_i \in K : i < \delta\}$ is an increasing chain. Then
        \begin{enumerate}
            \item $M_\delta := \bigcup_{i<\delta} M_i \in K$ and $M_i \leq_\K M_\delta$ for every $i<\delta$.
            \item Smoothness: If there is some $N\in K$ so that for all $i<\delta$ we have $M_i \leq_\K N$, then $M_\delta \leq_\K N$.
        \end{enumerate}
        \item L\"{o}wenheim-Skolem-Tarski axiom: There exists a cardinal $\lambda \geq |L(\K)|+\aleph_0$ such that for any $M\in K$ and $A \subseteq |M|$, there is some $M_0 \leq_\K M$ such that $A \subseteq |M_0|$ and $\|M_0\| \leq |A|+\lambda$. We write $\LS(\K)$ for the minimal such cardinal.
    \end{enumerate}
    
\end{defin}

\begin{remark}
In this paper, $L$ will be the language $L= \{ +, -, 0\}$ of abelian groups, $K$ will be a class of torsion-free abelian groups, and $\leq_\K$ will strengthen the pure subgroup relation $\leq_p$ of abelian groups.
\end{remark}

We say that $f: M \to N$ is a \emph{$\K$-embedding} if $f: M \cong f[M]$ and $f[M] \leq_{\K} N$. Observe that $\K$-embeddings are always injective. For $\lambda$ an infinite cardinal, let $K_\lambda := \{ M \in K : \| M \| = \lambda \}$.

Three common properties of AECs are no maximal models, joint embedding, and amalgamation. We say that $\K$ has \emph{no maximal models} if for every $M\in K$, there exists an $N\in K$ with $M \leq_{\K} N$ and $|M| \subsetneq |N|$. $\K$ has \emph{joint embedding} if for every $M_1, M_2 \in K$, there exist $N \in K$ and $\K$-embeddings $f_1: M_1 \to N$ and $f_2: M_2 \to N$. Lastly, $\K$ has \emph{amalgamation} if for every pair of $\K$-embeddings $f_1: M \to N_1$ and $f_2: M \to N_2$ there are $N\in K$ and $\K$-embeddings $g_1: N_1 \to N$ and $g_2: N_2 \to N$ such that $g_1\circ f_1=g_2\circ f_2$ on $M$.
 
An additional property we leverage in this paper is admitting intersections. This strong property often fails even for elementary classes. These AECs were introduced in \cite{bash}.

\begin{defin}\label{Def:Admit_Int}
    An AEC $\K$ \emph{admits intersections} if for every $N \in K$ and $A \subseteq |N|$,  $\Cl^{N}_{\K}(A):= \bigcap\{M : M \in K  \text{ and } A \subseteq M \leq_{\K} N \} \in K$ and $\Cl^{N}_{\K}(A) \leq_{\K} N$. 
\end{defin}

Galois types are the correct generalization of first-order types to AECs and were first introduced by Shelah in \cite{sh300}.

\begin{defin}\label{Def:Galois_type}
    Let $\K$ be an AEC and $\K^3$ be the set of triples of the form $(b, A, N)$, where $N \in K$, $A \subseteq |N|$, and $b \in N$.
    \begin{enumerate}
        \item For $(b_1, A_1, N_1), (b_2, A_2, N_2) \in \K^3$, we say that $(b_1, A_1, N_1)E^{\K}_{at} (b_2, A_2, N_2)$ if $A := A_1 = A_2$, and there exist $N \in K$ and $\K$-embeddings $f_i : N_i \to N$ for $i \in \{1, 2\}$ such that $f_1 \restr A = f_2\restr A = \id_A$ and $f_1(b_1) = f_2(b_2)$.
        \item Let $E^{\K}$ be the transitive closure of the relation $E^{\K}_{at}$. One may easily verify that $E^{\K}$ is an equivalence relation.
        \item For $(b, A, N) \in \K^3$, the \emph{Galois type} of $b$ over $A$ in $N$, denoted by $\gtp_{\K} (b / A; N)$, is the $E^{\K}$-equivalence class of $(b, A, N)$.
        \item For $M \in K$, let $\gS_{\K}(M) := \{ \gtp_{\K}(b/M; N) : M \leq_{\K} N \text{ and } b \in N\}$. 
        \item For $p = \gtp_{\K}(b/M; N) \in \gS_\K(M)$ and $A \subseteq |M|$, let $p\restr A := \gtp_{\K}(b/A; N)$. 
    \end{enumerate}
\end{defin}

For an AEC that admits intersections, there is a nicer characterization of Galois types.

\begin{fact}[{\cite[Proposition 2.18]{vaseyu}}]\label{Prop:Galois_type_Vasey}
    Let $\K$ be an AEC that admits intersections. Let $N_i \in K$, $A \subseteq N_i$ and $b_i \in N_i$ for $i\in \{1,2\}$. The following are equivalent:
    
    \begin{enumerate}
    \item $\gtp_{\K} (b_1 / A; N_1) = \gtp_{\K} (b_2 / A; N_2) $.
        \item There exists $f : \Cl_{\K}^{N_1}(\{b_1\} \cup A) \cong \Cl_{\K}^{N_2}(\{b_2\} \cup A)$ such that $f \restr A =\id_A$ and $f (b_1)=b_2$.
    \end{enumerate}
\end{fact}

Since Galois types are semantic objects, they might not be determined by their restriction to subsets of a given size. This leads to the notion of being tame.

\begin{defin}\label{Def:Tame}

Let $\lambda, \theta$ be infinite cardinals. $\K$ is \emph{$(<\lambda, \theta)$-tame} if for every $M \in K$ with $\| M \| = \theta$ and $p, q \in \gS_\K(M)$, if $p \neq q$, then there is $A \subseteq |M|$ such that $|A|< \lambda$ and $p \restr A \neq q \restr A$.

    We say that $\K$ is \emph{$(<\lambda)$-tame} if $\K$ is $(<\lambda, \theta)$-tame for every infinite cardinal $\theta$. We say $\K$ \emph{$\lambda$-tame} if it is $(<\lambda^+)$-tame.
\end{defin}

Stability is a key model-theoretic dividing line. Although it is not the main focus of this paper, we will study it for the classes we introduce in the next sections.

\begin{defin}\
\begin{enumerate}
    \item $\K$ is \emph{stable in $\lambda$} if  for any $M \in K$ with $\|M\|=\lambda$ it holds that $| \gS_{\K}(M) | \leq~\lambda$.
    
    We say that $\K$ is \emph{stable} if there exists a cardinal $\lambda \geq \LS(\K)$ for which $\K$ is stable in $\lambda$.

    \item $($\cite[Definition 2.4]{ps}$)$ We say that $\K$ is \emph{almost stable in $\lambda$}  if for any $M \leq_{\K} N$  with $\|M\| = \lambda$ it holds that $|\gS_{\K}(M;N)| \leq \lambda$, where $\gS_{\K}(M;N) := \{\gtp_{\K}(b/M; N) : b\in N\}$.

       We say that $\K$ is \emph{almost stable} if there exist arbitrarily large cardinals $\lambda \geq \LS(\K)$ for which $\K$ is almost stable in $\lambda$.
\end{enumerate}
    
\end{defin}

\begin{remark}
   If $\K$ has amalgamation. then $\K$ is almost stable in $\lambda$ if and only if  $\K$ is stable in $\lambda$, see \cite[Observation 2.6]{ps}. 
\end{remark}

These are all the notions of AECs used in this paper. Detailed introductions to abstract elementary classes from an algebraic perspective are given in \cite{bon25} and \cite[Section 2]{maztor}.

\subsection*{Torsion-free groups with pure embeddings}
We will denote the abstract elementary class of torsion-free abelian groups with pure embeddings by $\K_{TF} = (TF, \leq_p)$. This AEC is well-understood and we will leverage its properties throughout this paper.


\begin{fact}\label{Thm:Properties_TF}
    $\K_{TF}$ is an abstract elementary class with $\LS(\K_{TF}) = \aleph_0$ such that:
    \begin{enumerate}
        \item $\K_{TF}$ has no maximal models, joint embedding, and amalgamation.
        \item $\K_{TF}$ admits intersections.
        \item $($\cite[Lemma 4.6]{maz}$)$ $\K_{TF}$ is $(<\aleph_0)$-tame.
        \item $($\cite[Theorem 0.3]{baldwine}$)$ $\K_{TF}$ is stable in $\lambda$ if and only if $\lambda^{\aleph_0}=\lambda$.
    \end{enumerate}
\end{fact}

\section{Failure of $(<\aleph_0)$-tameness}\label{sectio-3}

We define the first abstract elementary class of interest in this paper.

\begin{defin}\label{Def:Kl(p_1,p_2)}
    Let $p_1, p_2$ be distinct primes. We define $\Kl=\Kl(p_1, p_2) = (K_1, \leq_{\Kl})$ as follows:
    \begin{enumerate}
        \item $K_1$ is the class of torsion-free abelian $G$ such that $|p_1^\omega G|, |p_2^\omega G| \leq \aleph_0$.
        \item For $G_1, G_2 \in K_1$, we have that $G_1 \leq_{\Kl} G_2$ if
        \begin{enumerate}
            \item $G_1 \leq_p G_2$,
            \item $p_1^\omega G_1 = p_1^\omega G_2$, and
            \item either $G_1 = p_1^\omega G_1$ or $p_2^\omega G_1 = p_2^\omega G_2$.
        \end{enumerate}
    \end{enumerate}
\end{defin}

The objective of this section is to show that this class of torsion-free abelian groups is an AEC such that the following properties hold.

\begin{theorem}\label{Thm:main_Kl}
    $\Kl=\Kl(p_1, p_2) = (K_1, \leq_{\Kl})$ is an abstract elementary class with $\LS(\Kl)=\aleph_0$ such that:
    \begin{enumerate}
        \item $\Kl$ is not $(<\aleph_0)$-tame.
        \item $\Kl$ is $\aleph_0$-tame.
        \item $\Kl$ is stable in $\lambda$ if and only if $\lambda^{\aleph_0} = \lambda$.
        \item $\Kl$ has no maximal models, but fails joint embedding and amalgamation.
        \item $\Kl$ admits intersections.
    \end{enumerate}
\end{theorem}

We show that $\Kl$ is an AEC after the following propositions.

\begin{prop}\label{Lem:Lemma1_Kl}
    Let $G_0 \leq_{\Kl} G_1$. If $G_1 = p_1^\omega G_1$, then $G_0=G_1$.
\end{prop}

\begin{proof}
Since $G_0 \leq_{\Kl} G_1$,   $p_1^\omega G_0 = p_1^\omega G_1$, so $G_1 = p_1^\omega G_1 \subseteq G_0$ and $G_0=G_1$.
\end{proof}

\begin{prop}\label{Lem:Transitivity_Coherence_Kl}
    The partial order $\leq_{\Kl}$ of $\Kl$ is transitive and satisfies coherence. 
    
\end{prop}
\begin{proof}

We show transitivity first. Suppose $G_0 \leq_{\Kl} G_1 \leq_{\Kl} G_2$. Then $G_0 \leq_p G_2$ by the transitivity of $\leq_p$. Further, $p_1^\omega G_0 = p_1^\omega G_1 = p_1^\omega G_2$. We are left to check Condition (2)(c) of Definition \ref{Def:Kl(p_1,p_2)}.
    
    If $G_0=p_1^\omega G_0$, then by the definition of $\leq_{\Kl}$ we have $G_0 \leq_{\Kl} G_2$. Otherwise, $p_2^\omega G_0 = p_2^\omega G_1$. Since $G_1 \leq_{\Kl} G_2$, either $G_1 = p_1^\omega G_1$ or $p_2^\omega G_1 = p_2^\omega G_2$. In the former case, by Proposition \ref{Lem:Lemma1_Kl}, $G_0=G_1$, so $G_0 \leq_{\Kl} G_2$. In the latter case, we have $p_2^\omega G_0 = p_2^\omega G_1 = p_2^\omega G_2$, so $G_0 \leq_{\Kl} G_2$ and $\leq_{\Kl}$ is transitive.

    We show coherence next. Suppose $G_0, G_1, G_2 \in K_1$ satisfy $G_0 \leq_{\Kl} G_2$, $G_1 \leq_{\Kl} G_2$, and $G_0 \subseteq G_1$. By the coherence of $\leq_p$, we have $G_0 \leq_p G_1$. Further, $p_1^\omega G_0 = p_1^\omega G_2 = p_1^\omega G_1$. We check Condition (2)(c) of Definition \ref{Def:Kl(p_1,p_2)} by dividing into  three cases:

    \underline{Case 1:} $G_0 = p_1^\omega G_0$. Then $G_0 \leq_{\Kl} G_1$ by the definition of $\leq_{\Kl}$.

    \underline{Case 2:}  $G_1 = p_1^\omega G_1$. Then $G_1 = p_1^\omega G_1 = p_1^\omega G_2 = p_1^\omega G_0 \subseteq G_0$ by Condition (2)(b) of Definition \ref{Def:Kl(p_1,p_2)}. Hence $G_0 \leq_{\Kl} G_1$.

    \underline{Case 3:} $G_0 \neq p_1^\omega G_0$ and  $G_1 \neq p_1^\omega G_1$. Then $p_2^\omega G_0 = p_2^\omega G_2=  p_2^\omega G_1$. Hence $G_0 \leq_{\Kl} G_1$. 
\end{proof}

We now prove $\Kl$ is an AEC.

\begin{lemma}\label{Prop:Kl_is_AEC}
    $\Kl$ is an AEC with $\LS(\Kl)=\aleph_0$.
\end{lemma}
\begin{proof}
    It is easy to check Properties (1) through (4) of Definition \ref{Def:AEC}. We have (5) (coherence) by Proposition \ref{Lem:Transitivity_Coherence_Kl}. We check (6) (Tarski-Vaught axioms) and (7) (L\"owenheim-Skolem-Tarski axiom).

    \underline{Tarski-Vaught Axioms}: Let $\{G_i : i<\delta\}$ be an increasing continuous chain in $K_1$ for $\delta$ a limit ordinal. Let $G_\delta = \bigcup_{i<\delta} G_i$. Observe that $G_\delta$ is a torsion-free abelian group and $G_i \leq_p G_\delta$ for all $i<\delta$. We check Conditions (2)(b) and (2)(c) of Definition \ref{Def:Kl(p_1,p_2)} to show that $G_i \leq_{\Kl} G_\delta$ and then use these to show that $G_\delta \in K_1$.

    We check (2)(b). Clearly $p_1^\omega G_i \subseteq p_1^\omega G_\delta$, so we show the other inclusion. Let $g\in p_1^\omega G_\delta$. Then $g\in G_j$ for some $j<\delta$ with $j>i$. Since $G_j \leq_p G_\delta$,  $g\in p_1^\omega G_j$. Since  $G_i \leq_{\Kl} G_j$, $p_1^\omega G_i= p_1^\omega G_j$. Hence $g\in p_1^\omega G_i$.
    
   We check (2)(c). If $G_i = p_1^\omega G_i$ there is nothing to show, so suppose $G_{i} \neq p_1^\omega G_{i}$. Then for every $j \geq i$, $p_2^ \omega G_i = p_2^\omega G_j$ because $G_i \leq_{\Kl} G_j$. Then the same argument as that above shows that $p_2^ \omega G_i  = p_2^ \omega G_\delta$.

   We show that $G_\delta \in K_1$. It follows from (2)(b) that $p_1^\omega G_\delta = p_1^\omega G_0$, so $|p_1^\omega G_\delta| = |p_1^\omega G_0| \leq \aleph_0$. To check the condition for $p_2$ we consider two cases. If $G_i = p_1^\omega G_i$ for all $i<\delta$, then $p_2^\omega G_\delta \subseteq  G_\delta = \bigcup_{i<\delta} G_i = \bigcup_{i<\delta} p_1^\omega G_i=p_1^\omega G_0$ where the last equality follows from $G_0 \leq_{\Kl} G_i$ for every $i < \delta$.\footnote{Observe that in this case $G_0 = G_i$ for all $i <\delta$ by Proposition \ref{Lem:Lemma1_Kl}. Thus $G_\delta = \bigcup_{i<\delta} G_i = G_0$.} If instead there is some minimal $i_0$ such that $G_{i_0} \neq p_1^\omega G_{i_0}$, then $p_2^\omega G_\delta = p_2^\omega G_{i_0}$ by (2)(c), so $|p_2^\omega G_\delta| = |p_2^\omega G_{i_0}| \leq \aleph_0$.

    \underline{Smoothness}: Further suppose there is some $H\in K_1$ such that $G_i \leq_{\Kl} H$ for all $i<\delta$. Clearly $G_\delta \leq_p H$. Moreover,  $p_1^\omega G_i = p_1^\omega H$ for all $i<\delta$, and hence $p_1^\omega G_\delta = p_1^\omega H$. So we are left to check Condition (2)(c).
    
    If $G_i = p_1^\omega G_i$ for all $i<\delta$, then we have shown $G_\delta = G_0 = p_1^\omega G_0 =p_1^\omega G_\delta$, so $G_\delta \leq_{\Kl} H$. If instead there is some minimal $i_0$ such that $G_{i_0} \neq p_1^\omega G_{i_0}$, then $p_2^\omega G_\delta = p_2^\omega G_{i_0} = p_2^\omega H$, so $G_\delta \leq_{\Kl} H$.

    \underline{L\"owenheim-Skolem-Tarski Axiom}: Given $G\in K_1$ and $A \subseteq G$, let $H$ be the pure closure of $A \cup p_1^\omega G \cup p_2^\omega G$ in $G$. Note that $\|H\| \leq |A \cup p_1^\omega G \cup p_2^\omega G| + \aleph_0 \leq |A| + \aleph_0$. One can check that $H$ is as needed as $p_1^\omega H=p_1^\omega G$ and $p_2^\omega H = p_2^\omega G$. 
\end{proof}

\begin{prop}\label{Prop:Kl_NMM_noJEP_noAP}
    $\Kl$ has no maximal models, but fails joint embedding and amalgamation.
\end{prop}
\begin{proof}
    To see $\Kl$ has no maximal models, let $G\in K_1$. Consider $G\oplus \Z$, and note that $G\oplus \Z \in K_1$ and $G \leq_{\Kl} G\oplus \Z$.
    
    To see $\Kl$ does not have joint embedding, consider $G_1 = \Z[1/p_1]$ and $G_2 = \Z$ and observe that $p_1^\omega G_1 = G_1 = \Z[1/p_1]\not\cong 0 = p_1^\omega G_2$.

    To see $\Kl$ does not have amalgamation, let $G_0 = \Z[1/p_1]$, $G_1 = \Z[1/p_1] \oplus \Z$, and $G_2 = \Z[1/p_1] \oplus \Z[1/p_2]$. Observe $G_0, G_1, G_2 \in K_1$ and $G_0 \leq_{\Kl} G_1, G_2$. Assume for the sake of a contradiction that there are $H\in K_1$ and $\Kl$-embeddings $f_\ell:G_\ell \to H$ for $\ell\in\{1,2\}$ such that $f_1\restr G_0 = f_2 \restr G_0$. Since $G_\ell \neq p_1^\omega G_\ell$ and $f_\ell$ is a $\Kl$-embedding for $\ell\in\{1,2\}$, we have that  $0 = p_2^\omega G_1 \cong p_2^\omega H \cong p_2^\omega G_2 = \Z[1/p_2]$, but this is a contradiction.
\end{proof}

\begin{prop}\label{Prop:Kl_admit_int}
    $\Kl$ admits intersections.
\end{prop}
\begin{proof}
    It suffices to prove that for any $\{H_i : i < \lambda\}$ with $H_i \leq_{\Kl} G$ for all $i < \lambda$, we have that $\bigcap_{i < \lambda} H_i \in K_1$ and $\bigcap_{i < \lambda} H_i \leq_{\Kl} G$. We will often use that $p_\ell^\omega (\bigcap_{i < \lambda} H_i) = \bigcap_{i < \lambda} p_\ell^\omega H_i$ for $\ell\in\{1,2\}$, which follows from Fact \ref{Facts:Uncited}.
    
    Note that $\bigcap_{i < \lambda} H_i$ is a torsion-free abelian group with $p_\ell^\omega (\bigcap_{i < \lambda} H_i) = \bigcap_{i < \lambda} p_\ell^\omega H_i \subseteq p_\ell^\omega G$ for $\ell\in\{1,2\}$, so $\bigcap_{i < \lambda} H_i \in K_1$ as $G\in K_1$. Moreover, observe that $\bigcap_{i < \lambda} H_i \leq_p G$ and $p_1^\omega(\bigcap_{i < \lambda}  H_i) = p_1^\omega G$. So we are left to check Condition (2)(c). We consider two cases.

    If $p_2^\omega H_i = p_2^\omega G$ for all $i < \lambda$, then $p_2^\omega \left(\bigcap_{i<\lambda} H_i\right) = p_2^\omega G$. If instead, there is some $i_0<\lambda$ such that $p_2^\omega H_{i_0} \neq p_2^\omega G$, we must then have $H_{i_0} = p_1^\omega H_{i_0}$. It is enough to show that $\bigcap_{i<\lambda} H_i = \bigcap_{i<\lambda} p_1^\omega  H_i$.
    
    Let $h\in \bigcap_{i<\lambda} H_i$ and fix some $i < \lambda$. Then $h \in H_{i_0} = p_1^\omega H_{i_0} \subseteq p_1^\omega G$. Since $h \in H_i \cap p_1^\omega G$ and $H_i \leq_p G$, $h \in p_1^\omega H_{i}$. \end{proof}

  We are actually able to get a very simple description of the closure operator in $\Kl$.   
    
\begin{lemma}\label{Lem:Kl_closure_equal_to_TF}
    Let $G \leq_{\Kl} H$ and $a\in H \backslash G$.
    \begin{enumerate}
        \item If $G= p_1^\omega G$, then $\Cl_{\Kl}^H(\{a\} \cup G) = \Cl_{\K_{TF}}^H(\{a\} \cup G \cup p_2^\omega H)$.
        \item  If $G \neq p_1^\omega G$,  then $\Cl_{\Kl}^H(\{a\} \cup G) = \Cl_{\K_{TF}}^H(\{a\} \cup G)$.
    \end{enumerate}
\end{lemma}
\begin{proof}
    Observe first that for every $A \subseteq H$, $\Cl_{\Kl}^H(A) = \bigcap\{G_i\in K_1 : A \subseteq G_i, G_i \leq_{\Kl} H\} \supseteq \bigcap\{G_i\in TF : A \subseteq G_i, G_i \leq_{p} H\} = \Cl_{\K_{TF}}^H(A)$.
    
    \begin{enumerate}
        \item Since $a \not\in G$, we have that $p_1^\omega \Cl_{\Kl}^H(\{a\}\cup G) \neq \Cl_{\Kl}^H(\{a\}\cup G)$ as otherwise $a \in \Cl_{\Kl}^H(\{a\}\cup G) = p_1^\omega \Cl_{\Kl}^H(\{a\}\cup G) = p_1^\omega H = p_1^\omega G = G$ where the second and third equality follow from the definition of $\leq_{\Kl}$. Hence   $p_2^\omega \Cl_{\Kl}^H(\{a\}\cup G) = p_2^\omega H$ by Condition (2)(c). Thus $\{a\} \cup G \cup p_2^\omega H \subseteq \Cl_{\Kl}^H(\{a\}\cup G)$. So we have $\Cl_{\Kl}^H(\{a\}\cup G) = \Cl_{\Kl}^H(\{a\}\cup G \cup p_2^\omega H) \supseteq \Cl_{\K_{TF}}^H(\{a\}\cup G \cup p_2^\omega H)$.
    
        On the other hand, observe that $\Cl_{\K_{TF}}^H(\{a\}\cup G \cup p_2^\omega H) \leq_{\Kl} H$ and $\{a\}\cup G \subseteq \Cl_{\K_{TF}}^H(\{a\}\cup G \cup p_2^\omega H)$. Then by the minimality of the $\Kl$ closure, $\Cl_{\Kl}^H(\{a\}\cup G) \subseteq \Cl_{\K_{TF}}^H(\{a\}\cup G \cup p_2^\omega H)$, and  $\Cl_{\Kl}^H(\{a\}\cup G) = \Cl_{\K_{TF}}^H(\{a\}\cup G \cup p_2^\omega H)$ follows.
     
        \item Since $G \neq p_1^\omega G$, then $p_2^\omega G = p_2^\omega H$ by Condition (2)(c). So $\Cl_{\K_{TF}}^H(\{a\} \cup G) \leq_{\Kl} H$. Then by the minimality of the $\Kl$ closure, $\Cl_{\Kl}^H(\{a\}\cup G) \subseteq \Cl_{\K_{TF}}^H(\{a\}\cup G)$. The other inclusion follows from the observation made at the top of the proof. Hence $\Cl_{\Kl}^H(\{a\} \cup G) = \Cl_{\K_{TF}}^H(\{a\} \cup G)$. \qedhere
    \end{enumerate}   
\end{proof}

 The following groups will be used to witness the failure of $(<\aleph_0)$-tameness.

\begin{defin} \label{def:M,N}
    Let $U := \prod_{\alpha < \omega} \Q e_\alpha$ and let $M := \bigoplus_{\alpha < \omega} \Z[1/p_1, 1/p_2] e_\alpha \leq U$. Occasionally, we may use the notation $(a_0,a_1,a_2,\ldots)$ to refer to the element $a = \sum_{\alpha < \omega} a_\alpha e_\alpha\in U$ with $a_\alpha \in \Q$.

    Let $q_1, q_2 \not\in \{p_1, p_2\}$ be distinct primes. For all $\alpha < \omega$, pick $c_\alpha, d_\alpha \in \Z$ such that $c_\alpha q_1^\alpha + d_\alpha q_2^\alpha = 1$.

    Let 
    \[ y_\alpha := \sum_{k \geq \alpha} c_k q_1^{k-\alpha}e_{k} = (0, \ldots, 0, c_\alpha q_1^0, c_{\alpha+1} q_1^1, c_{\alpha+2} q_1^2, \ldots )  \]
    and 
    \[ z_\alpha := \sum_{k \geq \alpha} d_k q_2^{k-\alpha}e_{k} = (0, \ldots, 0, d_\alpha q_2^0, d_{\alpha+1} q_2^1, d_{\alpha+2} q_2^2, \ldots ) \]

    Let $N := \left\langle M, \Z[1/p_2] y_\alpha, \Z[1/p_2] z_\alpha : \alpha < \omega \right\rangle \leq U$ in $TF$. 
\end{defin}

\begin{remark}\label{Remark:easy}
    Let $M$ and $N$ be as in Definition \ref{def:M,N}.
    \begin{enumerate}
        \item  $q_1^\omega N = q_2^\omega N = 0$.
        \item  $y_0+z_0=(c_0 q_1^0 + d_0 q_2^0,c_1 q_1^1 + d_1 q_2^1,c_2 q_1^2 + d_2 q_2^2,\ldots)=(1,1,1,\ldots) \in N$.
        \item If $n \in N$, then there are $m \in M$, $\beta < \omega$ and $r,s \in \Z[1/p_2]$ such that $n = m + r y_\beta + s z_\beta$. This follows from the fact that if $\beta = \alpha + k$ for some $k < \omega$, then $y_\alpha = c_\alpha e_\alpha + c_{\alpha+1} q_1 e_{\alpha +1} + \ldots + c_{\alpha+k-1} q_1^{k-1} e_{\alpha+k-1}  + q_1^k y_\beta$ and  $c_\alpha  e_\alpha + c_{\alpha+1} q_1 e_{\alpha +1}  + \ldots + c_{\alpha+k-1} q_1^{k-1} e_{\alpha+k-1} \in M$. A similar argument can be made for the $z_\alpha$'s.  
    \end{enumerate}
\end{remark}

\begin{prop}\label{Prop:div-MN}
    Let $M, N$ be as in Definition \ref{def:M,N}.
    \begin{enumerate}
        
        \item If $r,s \in \Z[1/p_2]$ and $r y_\alpha + s z_\alpha \in p_1^\omega N$ for $\alpha <\omega$, then $r=s=0$.
        \item $p_1^\omega M = M$,  $p_2^\omega N = N$, and $p_1^\omega N = M$.
        \item $M, N \in K_1$ and $M \leq_{\Kl} N$.
    \end{enumerate}
\end{prop}
\begin{proof}\
    \begin{enumerate}
        \item Suppose $r y_\alpha + s z_\alpha \in p_1^\omega N$. Then for all $n<\omega$, $r y_\alpha + s z_\alpha = p_1^n h_n$ for $h_n\in N$.
        
        For the following argument, fix any $n<\omega$. By Remark \ref{Remark:easy} we may assume that $h_n = m_n + r_n y_\beta + s_n z_\beta$ for some $\alpha \le \beta < \omega$, $m_n \in M$, and $r_n, s_n \in \Z[1/p_2]$.

        Let $\beta = \alpha + k$ for some $k < \omega$. Then by Remark \ref{Remark:easy}, $y_\alpha = m_{y,\alpha,k} + q_1^k y_\beta$ and $z_\alpha = m_{z,\alpha,k} + q_2^k z_\beta$ for $m_{y,\alpha,k}, m_{z,\alpha,k} \in M$. Thus we obtain the equation
        \[r(m_{y,\alpha,k} + q_1^k y_\beta) + s(m_{z,\alpha,k} + q_2^k z_\beta) = p_1^n m_n + p_1^n r_n y_\beta + p_1^n s_n z_\beta. \]    
        So we have
        \[r m_{y,\alpha,k} + s m_{z,\alpha,k} - p_1^n m_n = (p_1^n r_n - r q_1^k)y_\beta + (p_1^n s_n - s q_2^k)z_\beta \in M.\]

        Since $(p_1^n r_n - r q_1^k)y_\beta + (p_1^n s_n - s q_2^k)z_\beta \in M$, there exists some $j<\omega$ such that for all $j\le i<\omega$, the $e_{\beta + i}$-coordinate of this expression is $0$. Thus $(p_1^n r_n - r q_1^k)c_{\beta+i} q_1^i + (p_1^n s_n - s q_2^k)d_{\beta+i} q_2^i = 0$ for all $i\ge j$ with $r, r_n, s, s_n \in \Z[1/p_2]$.

        But then $q_1^i | (p_1^n s_n - s q_2^k) d_{\beta+i} q_2^i$ in $\Z[1/p_2]$ for all $i\ge j$. Since $\gcd(q_1, q_2) = \gcd(q_1, d_{\beta+i})=1$, $q_1^i | (p_1^n s_n - s q_2^k)$ in $\Z[1/p_2]$. But this is true for arbitrarily large $i$, so $(p_1^n s_n - s q_2^k) \in q_1^\omega \Z[1/p_2] =0$. Thus $p_1^n s_n = s q_2^k$, and since $\gcd(p_1, q_2)=1$, $p_1^n | s$ in $\Z[1/p_2]$.
        
        Note that $p_1^n | s$ holds independently of our choice of $n<\omega$, so $s \in p_1^\omega \Z[1/p_2] = 0$, and $s=0$ follows. The argument for $r=0$ is similar.
    
        \item The first two equalities follow directly from the definition of $M$ and $N$, so we show only the last equality. Observe  $M = p_1^\omega M \subseteq p_1^\omega N$, so we show the other inclusion.
    
        Let $n\in p_1^\omega N$. Since $n \in N$, there exist $m\in M$, $r,s \in \Z[1/p_2]$, and $\alpha < \omega$ such that $n = m + r y_\alpha + s z_\alpha$ by Remark \ref{Remark:easy}. Since $M = p_1^\omega M \subseteq p_1^\omega N$, we have $n-m = r y_\alpha + s z_\alpha \in p_1^\omega N$. By (1), this implies that $n-m=0$.  Thus $n=m \in M$. 
     
        \item  Since $M, N$ are both countable, we have that $M, N \in K_1$. Furthermore, $M \leq_p N$ as $M = p_1^\omega N$, and Conditions (2)(b) and (2)(c) of Definition \ref{Def:Kl(p_1,p_2)} follow from (2). \qedhere
    \end{enumerate}  
\end{proof}

The reason tameness fails in our example is due to the following algebraic result.

\begin{theorem}\label{Prop:Homo_Unique}     Let $M, N$ be as in Definition \ref{def:M,N}.
    If $f:N \to N$ is a group homomorphism with $f(e_\alpha)=e_\alpha$ for all $\alpha<\omega$, then $f=\id_N$.
\end{theorem}
\begin{proof}
    Consider $g:= f - \id_N$. We show that $g=0$. Note that $g\restr M = 0$ is clear. Thus it suffices to prove that $g(y_\alpha) = g(z_\alpha) = 0$ for all $\alpha < \omega$.
    
    Let $\alpha <\omega$. We show that $g(y_\alpha)=0$; the proof that  $g(z_\alpha) = 0$ is similar. Recall that   $y_\alpha = (0, \ldots, 0, c_\alpha q_1^0, c_{\alpha+1} q_1^1, c_{\alpha+2} q_1^2, \ldots )= \sum_{k \geq \alpha} c_k q_1^{k-\alpha}e_{k}$.

Then for every $n < \omega$, we have that
\[ y_\alpha - c_\alpha e_\alpha - c_{\alpha+1}q_1e_{\alpha +1} - \ldots - c_{\alpha+ n - 1}q_1^{n-1}e_{\alpha+n - 1} = q_1^n\big( \sum_{k \geq \alpha + n} c_k q_1^{k-\alpha -n}e_{k}\big).\]

Applying $g$ and using that  $g\restr M = 0$, we get that
\[ g(y_\alpha)= q_1^n \Big( g\big(\sum_{k \geq \alpha + n} c_k q_1^{k-\alpha -n}e_{k}\big)\Big)\]
for every $n <\omega$. Thus we have that $q_1^n | g(y_\alpha)$ for every $n<\omega$, so $g(y_\alpha) \in q_1^\omega N = 0$. Hence $g(y_\alpha) = 0$. \end{proof}

We can now prove  one of the key results of the paper. 

\begin{theorem}\label{Thm:Kl_not_aleph0_tame}
    $\Kl$ is not $(<\aleph_0)$-tame.
\end{theorem}
\begin{proof}
    Let  $M, N$ be as in Definition \ref{def:M,N}. We have that $M, N \in K_1$ and $M\leq_{\Kl} N$ by Proposition \ref{Prop:div-MN}. Let $a = (1,1,1,1,\ldots) = y_0+z_0 \in |N|$ and $b = (2,1,1,1,\ldots) = e_0+y_0+z_0 \in |N|$. We show that $\gtp_{\Kl}(a/M; N), \gtp_{\Kl}(b/M; N)$ witness that $\Kl$ is not $(<\aleph_0)$-tame in the following two claims.

    \underline{Claim 1}: $\gtp_{\Kl}(a/M; N) \neq \gtp_{\Kl}(b/M; N)$.

    \underline{Proof of Claim 1}: Assume for the sake of contradiction that $\gtp_{\Kl}(a/M; N) = \gtp_{\Kl}(b/M; N)$. Then there is an isomorphism  $f:\Cl_{\Kl}^{N}(\{a\} \cup M) \to \Cl_{\Kl}^{N}(\{b\} \cup M)$ such that $f(a) = b$ and $f \restr M = \id_M$ by Fact \ref{Prop:Galois_type_Vasey}. Since $a, b \not\in M$, it follows from  Lemma \ref{Lem:Kl_closure_equal_to_TF} that $\Cl_{\Kl}^{N}(\{a\} \cup M)= \Cl_{TF}^{N}(\{a\} \cup M \cup p_2^\omega N)$ and $\Cl_{\Kl}^{N}(\{b\} \cup M)= \Cl_{TF}^{N}(\{b\} \cup M \cup p_2^\omega N)$. Since $p_2^\omega N = N$ by Proposition \ref{Prop:div-MN}, it follows that  $\Cl_{\Kl}^{N}(\{a\} \cup M)= N = \Cl_{\Kl}^{N}(\{a\} \cup M)$. Hence $f$ is an isomorphism from $N$ to $N$. 
    
    Since $f: N \to N$ and $f \restr M = \id_M$, we have that $f =\id_N$ by Theorem \ref{Prop:Homo_Unique}. This is a contradiction as $f(a)= b$. $\dagger_{\text{Claim 1}}$ \medskip

\underline{Claim 2} For all finite $C \subseteq M$,  $\gtp_{\Kl}(a/C; N) = \gtp_{\Kl}(b/C; N)$.

\underline{Proof of Claim 2}: Since $M$ is a direct sum and $C$ is finite, there is $n < \omega$ such that $C \subseteq \bigoplus_{\alpha < n} \Z[1/p_1, 1/p_2] e_\alpha$.

Let $f_a: N \to N$ be given by $f_a(x_0, x_1, x_2, x_3, \ldots) = (x_0+x_n, x_1, x_2, x_3, \ldots)$. Observe that $f_a$ is an isomorphism such that the following diagram commutes
    \[
    \begin{tikzcd}
        N \arrow[r, "f_a"] & N \\
        C \arrow[u, "\subseteq"] \arrow[r, "\subseteq"] & N \arrow[u, "\id_N"]
    \end{tikzcd}
    \]
    and $f_a(a)=b$. Hence $\gtp_{\Kl}(a/C; N) = \gtp_{\Kl}(b/C; N)$. $\dagger_{\text{Claim 2}}$\qedhere

\end{proof}

Although the previous result shows that Galois types in $\K_{TF}$ are, in general, distinct from those in $\Kl$, the next result shows that many times they are the same.

\begin{prop}\label{Prop:Kl_gtp_equal_to_TF}
    If $G \leq_{\Kl} H_1, H_2$, $a_1\in H_1 $, $a_2 \in H_2$, and $p_2^\omega H_1, p_2^\omega H_2 \subseteq G$, then $\gtp_{\K_{TF}}(a_1/G; H_1) = \gtp_{\K_{TF}}(a_2/G; H_2)$ if and only if $\gtp_{\Kl}(a_1/G; H_1) = \gtp_{\Kl}(a_2/G; H_2)$.
\end{prop}
\begin{proof} We may assume $a_1, a_2 \notin G$ as otherwise there is nothing to show. Observe that since $p_2^\omega H_1, p_2^\omega H_2 \subseteq G$, $\{a_i\} \cup G \cup p_2^\omega H_i = \{a_i\} \cup G$ for $i\in \{1,2\}$. Thus by Proposition \ref{Lem:Kl_closure_equal_to_TF}, $\Cl_{\Kl}^{H_i}(\{a_i\} \cup G) = \Cl_{\K_{TF}}^{H_i}(\{a_i\} \cup G)$. The result then follows from the fact that  $\K_{TF}$ and $\Kl$ admit intersections and Fact \ref{Prop:Galois_type_Vasey}.
\end{proof}

We now prove that $\Kl$  has a significant amount of tameness.

\begin{lemma}\label{Lem:Kl_aleph1_tame}
    $\Kl$ is $\aleph_0$-tame.
\end{lemma}
\begin{proof}
    Suppose $G \leq_{\Kl} H_1, H_2$, $a_1\in H_1 $, $a_2 \in H_2$, and $\gtp_{\Kl}(a_1/G; H_1) \restr C = \gtp_{\Kl}(a_2/G; H_2) \restr C$ for every countable $C\subseteq G$. We have two cases:

    \underline{Case 1}: $G=p_1^\omega G$. Then $G$ is countable by (1) of Definition \ref{Def:Kl(p_1,p_2)}, and $\gtp_{\Kl}(a_1/G; H_1) = \gtp_{\Kl}(a_1/G; H_1) \restr G= \gtp_{\Kl}(a_2/G; H_2) \restr G = \gtp_{\Kl}(a_2/G; H_2)$.

    \underline{Case 2}: $G \neq p_1^\omega G$. Since $G\leq_{\Kl} H_1, H_2$, we have that $ p_2^\omega H_1 = p_2^\omega H_2 = p_2^\omega G  \subseteq G$. Then by Proposition \ref{Prop:Kl_gtp_equal_to_TF}, it is enough to show $\gtp_{\K_{TF}}(a_1/G; H_1) = \gtp_{\K_{TF}}(a_2/G; H_2)$. Since $\K_{TF}$ is $(<\aleph_0)$-tame by Fact \ref{Thm:Properties_TF}, it suffices to prove that $\gtp_{\K_{TF}}(a_1/G; H_1) \restr D = \gtp_{\K_{TF}}(a_2/G; H_2) \restr D$ for all finite $D \subseteq G$.

    Let $D \subseteq G$ be finite. By assumption, $\gtp_{\Kl}(a_1/G; H_1) \restr D = \gtp_{\Kl}(a_2/G; H_2) \restr D$, so there exists $f:\Cl_{\Kl}^{H_1}(\{a_1\} \cup D) \cong \Cl_{\Kl}^{H_2}(\{a_2\} \cup D)$ with $f(a_1)=a_2$ and $f \restr D = \id_D$ by Fact \ref{Prop:Galois_type_Vasey}.

    Observe that $\Cl_{\Kl}^{H_1}(\{a_1\} \cup D) \leq_p H_1$ and $\Cl_{\Kl}^{H_2}(\{a_2\} \cup D) \leq_p H_2$, and both are torsion-free, so we have
    \begin{align*}
        (a_1, D, H_1) & E_{at}^{\K_{TF}} (a_1, D, \Cl_{\Kl}^{H_1}(\{a_1\} \cup D)) \\
        & E_{at}^{\K_{TF}} (f(a_1), f[D], f[\Cl_{\Kl}^{H_1}(\{a_1\} \cup D)]) = (a_2, D, \Cl_{\Kl}^{H_2}(\{a_2\} \cup D)) \\
        & E_{at}^{\K_{TF}} (a_2, D, H_2).
    \end{align*}
    Hence $\gtp_{\K_{TF}}(a_1/G; H_1) \restr D = \gtp_{\K_{TF}}(a_2/G; H_2) \restr D$.
\end{proof}

\begin{lemma}\label{Lem:Kl_aleph_theta_tame} For every  uncountable cardinal $\theta$,
$\Kl$ is $(<\aleph_0, \theta)$-tame.
\end{lemma}
\begin{proof}
Observe that if $G \in (K_1)_\theta$, then $p_1^\omega G \neq G$. Hence the proof of Case 2 of Lemma \ref{Lem:Kl_aleph1_tame} can be used to show that $\Kl$ is $(<\aleph_0, \theta)$-tame. 
\end{proof}

We characterize the stability spectrum of $\Kl$ using Proposition \ref{Prop:Kl_gtp_equal_to_TF}.

\begin{theorem}\label{Thm:Kl_stable}
       $\Kl$ is stable in $\lambda$ if and only if $\lambda^{\aleph_0} = \lambda$.
\end{theorem}
\begin{proof}
    $\Rightarrow$: We do the proof by contrapositive. Suppose $\lambda^{\aleph_0} \neq \lambda$. Then there is $\{ \gtp_{\K_{TF}}(b_i/A; B_i) : i <\lambda^+\}$  a set of distinct Galois types in $\K_{TF}$ such that $A$ is free and $\| A \| = \lambda$, and $B_i$ is free for all $i < \lambda^+$ by \cite[p. 37]{baldwine}. Since $p^\omega F = 0$ for every free abelian  group $F$ and prime number $p$, it follows that $A, B_i \in K_1$ and $A \leq _{\Kl} B_i$ for all $i < \lambda^+$. Furthermore,  it follows from Proposition \ref{Prop:Kl_gtp_equal_to_TF} that $\gtp_{\Kl}(b_i/A; B_i) \neq \gtp_{\Kl}(b_j/A; B_j)$ for all $i \neq j < \lambda^+$. Therefore, $|\gS_{\Kl}(A)| \geq \lambda^+$,  and $\Kl$ is not stable in $\lambda$.
    
   $\Leftarrow$: Suppose $\lambda^{\aleph_0} = \lambda$ and let $G \in K_1$ with $\|G\|=\lambda$. Assume for sake of contradiction that $|\gS_{\Kl}(G)| > \lambda$. Let $\{\gtp_{\Kl}(a_i/G; H_i) : i<\lambda^+\}$ be an enumeration of distinct Galois types in $\Kl$. Let $\Phi: \lambda^+ \to \gS_{\K_{TF}}(G)$ be given by $\Phi(i)=\gtp_{\K_{TF}}(a_i/G; H_i)$.

    As $|\gS_{\K_{TF}}(G)| \leq \lambda$ since $\K_{TF}$ is stable in $\lambda$ by Fact \ref{Thm:Properties_TF}, there exist $i\neq j <\lambda^+$ such that $\gtp_{\K_{TF}}(a_i/G; H_i) = \gtp_{\K_{TF}}(a_j/G; H_j)$. Observe that $p_1^\omega G \neq G$ since $\lambda >\aleph_0$. As $G \leq_{\Kl} H_i, H_j$, we have that $  p_2^\omega H_i = p_2^\omega H_j = p_2^\omega G \subseteq G$. Then $\gtp_{\Kl}(a_i/G; H_i) = \gtp_{\Kl}(a_j/G; H_j)$ by Proposition \ref{Prop:Kl_gtp_equal_to_TF}, a contradiction.
    \end{proof}

We conclude this section with the proof of Theorem \ref{Thm:main_Kl}.

\begin{proof}[Proof of Theorem \ref{Thm:main_Kl}]
    $\Kl$ is an AEC with $\LS(\Kl) = \aleph_0$ by Proposition \ref{Prop:Kl_is_AEC}.

    \begin{enumerate}
        \item Follows from Theorem \ref{Thm:Kl_not_aleph0_tame}.
        \item Follows from Lemma \ref{Lem:Kl_aleph1_tame}.
        \item Follows from Theorem \ref{Thm:Kl_stable}.
        \item Follows from Proposition \ref{Prop:Kl_NMM_noJEP_noAP}.
        \item Follows from Proposition \ref{Prop:Kl_admit_int}. \qedhere
    \end{enumerate}
\end{proof}

\section{Beyond $(<\aleph_0)$-tameness}\label{Section:Kk}

We define the second family of abstract elementary classes of interest in this paper.

\begin{defin}\label{Def:Kk(p_1,p_2)}
    Let $p_1, p_2$ be distinct primes and let $\kappa$ be an infinite cardinal. We define $\Kk=\Kk(\kappa)=\Kk(p_1, p_2, \kappa) = (K_2(\kappa), \leq_{\Kk(\kappa)})$ as follows:
    \begin{enumerate}
        \item $K_2(\kappa)$ is the class of torsion-free abelian $G$ such that $|p_1^\omega G|, |p_2^\omega G| \leq \kappa$.
        \item For $G_1, G_2 \in K_2(\kappa)$, we have that $G_1 \leq_{\Kk(\kappa)} G_2$ if
        \begin{enumerate}
            \item $G_1 \leq_p G_2$,
            \item $p_1^\omega G_1 = p_1^\omega G_2$, and
            \item either $p_2^\omega G_1 = p_1^\omega G_1$ or $p_2^\omega G_1 = p_2^\omega G_2$.
        \end{enumerate}
    \end{enumerate}
\end{defin}

\begin{remark}
Observe that the only differences between $\Kl$ (Definition \ref{Def:Kl(p_1,p_2)}) and $\Kk(\kappa)$ (Definition \ref{Def:Kk(p_1,p_2)}) are that in Condition (1) we changed $\aleph_0$ to $\kappa$ and that in Condition (2)(c) we changed $p_1^\omega G_1 = G_1$ to $p_1^\omega G_1 = p_2^\omega G_1$. The second change had to be implemented in order to show the failure of tameness for models above $LS(\Kk(\kappa))$ (see Remark \ref{Remark:up}).

Since $\Kl$ is similar to $\Kk(\kappa)$, some of the arguments from Section \ref{sectio-3} can be easily adapted to obtain results for $\Kk(\kappa)$, so we will present many results without a proof.
\end{remark}

\begin{remark}
Observe that $\leq_{\Kk(\kappa)}\,=\,\leq_{\Kk(\mu)}$ for all cardinals $\kappa, \mu$. Due to this we will always denote the relation by $\leq_{\Kk}$.

Although we only have that $\Kk(\kappa) \subseteq \Kk(\mu)$ if $\kappa < \mu$, many times we will simply write $\Kk$ instead of $\Kk(\kappa)$.
\end{remark}

The objective of this section is to show that this class of torsion-free abelian groups is an AEC such that the following properties hold.  Compare to Theorem~\ref{Thm:main_Kl}.

\begin{theorem}\label{Thm:main_Kk}
    Let $\kappa$ be an infinite cardinal. $\Kk =\Kk(\kappa) =  (K_2, \leq_{\Kk})$ is an abstract elementary class with $\LS(\Kk)= \kappa$ such that:
    \begin{enumerate}
        \item If $\kappa = 2^\mu$ for some regular uncountable cardinal $\mu$ less than the first measurable cardinal, then $\Kk(\kappa)=\Kk(2^\mu)$ is not $(<\mu)$-tame.
        \item If $\lambda^{\aleph_0} = \lambda$ and $\lambda \ge \kappa$, then $\Kk$ is almost stable in $\lambda$.
        \item $\Kk$ has no maximal models, but 
        fails joint embedding and amalgamation.
        \item $\Kk$ admits intersections.
    \end{enumerate}
\end{theorem}

\begin{remark}\label{Remark:Kk_AEC_Similar_Kl} The proofs that $\Kk$ is an AEC and that
  (3) and (4) of Theorem~\ref{Thm:main_Kk} hold are analogous to those of Lemma \ref{Prop:Kl_is_AEC}, Proposition \ref{Prop:Kl_NMM_noJEP_noAP}, and Proposition \ref{Prop:Kl_admit_int}, respectively. \end{remark}

As in Section \ref{sectio-3} it is possible to obtain a simple description of the closure operator in $\Kk$. 

\begin{prop}\label{Prop:Kk_closure}
    Let $G \leq_{\Kk} H$ and $a\in H\setminus G$.
    \begin{enumerate}
        
        \item If $p_1^\omega \Cl_{\K_{TF}}^H(\{a\} \cup G) \neq p_2^\omega \Cl_{\K_{TF}}^H(\{a\} \cup G)$, then $\Cl_{\Kk}^H(\{a\} \cup G) = \Cl_{\K_{TF}}^H(\{a\} \cup G \cup p_2^\omega H)$. 
        \item If $p_1^\omega \Cl_{\K_{TF}}^H(\{a\} \cup G) = p_2^\omega \Cl_{\K_{TF}}^H(\{a\} \cup G)$, then $\Cl_{\Kk}^H(\{a\} \cup G) = \Cl_{\K_{TF}}^H(\{a\} \cup G)$. 
    \end{enumerate}   
In particular, if $a\in p_2^\omega H$, then $\Cl_{\Kk}^H(\{a\} \cup G) = \Cl_{\K_{TF}}^H(\{a\} \cup G \cup p_2^\omega H)$.
\end{prop}
\begin{proof} The proofs of (1) and (2) are analogous to those of Lemma \ref{Lem:Kl_closure_equal_to_TF}.

For the \emph{in particular}, observe that if $a\in p_2^\omega H$  then $p_1^\omega \Cl_{\K_{TF}}^H(\{a\} \cup G) \neq p_2^\omega \Cl_{\K_{TF}}^H(\{a\} \cup G)$ as otherwise $a \in G$ because $a \in p_2^\omega \Cl_{\K_{TF}}^H(\{a\} \cup G) = p_1^\omega \Cl_{\K_{TF}}^H(\{a\} \cup G) \subseteq p_1^\omega H =  p_1^\omega G \subseteq G$ where the first equality is by assumption and the second by definition of $\leq_{\Kk}$. Hence the result follows by (1). \end{proof}


  Throughout the rest of the section we will mostly work in $\Kk (2^\mu)=\Kk(p_1, p_2, 2^\mu)$ for a fixed regular uncountable cardinal $\mu$. The following groups will be used to witness the failure of $(<\mu)$-tameness.

\begin{defin}\label{Def:K2_M_N}
    Let $U_\mu := \prod_{\alpha < \mu} \Q e_\alpha$, $M_\mu := \bigoplus_{\alpha < \mu} \Z[1/p_1, 1/p_2] e_\alpha \leq U_\mu$, $N_\mu := \left\langle M_\mu, \prod_{\alpha<\mu} \Z[1/p_2] e_\alpha \right\rangle \leq U_\mu$, and $V_\mu := p_1^\omega N_\mu$.
\end{defin}

\begin{prop}\label{Prop:Basic-Nmu}
Let $M_\mu, N_\mu, V_\mu$ as in Definition \ref{Def:K2_M_N}.
\begin{enumerate}
    \item Let $(a_\alpha)_{\alpha < \mu} \in N_\mu$. Then $(a_\alpha)_{\alpha < \mu}\in p_1^\omega N_\mu = V_\mu$ if and only if for every $n<\omega$, there is $I_n \subseteq \mu$ finite such that $a_\alpha \in p_1^n \Z[1/p_2]$  for all $\alpha \in \mu \backslash I_n$. In particular, if $(a_\alpha)_{\alpha < \mu}\in p_1^\omega N_\mu$, then  $a_\alpha=0$ with at most countably many exceptions.
   \item $p_2^\omega N_\mu = N_\mu$ and $p_1^\omega V_\mu = p_2^\omega V_\mu = V_\mu$.
        \item $V_\mu, N_\mu \in K_2(2^\mu)$ and $V_\mu \leq_{\Kk} N_\mu$.
\end{enumerate}
\end{prop}
\begin{proof} \
\begin{enumerate}
    \item $\Rightarrow$: By definition, $(a_\alpha)_{\alpha < \mu}\in p_1^\omega N_\mu = V_\mu$ exactly when, for every $n<\omega$, there exist $m_n \in M_\mu$ and $(b_\alpha)_{\alpha<\mu} \in \prod_{\alpha<\mu} \Z[1/p_2] e_\alpha$ such that $(a_\alpha)_{\alpha<\mu} = p_1^n m_n + (p_1^n b_\alpha)_{\alpha < \mu}$. Noting that the $e_\alpha$-coordinate of $p_1^n m_n$ is zero for all but finitely many $\alpha$ as $M_\mu$ is a direct sum, the result follows. 

    $\Leftarrow$:  Since $p_1^\omega \Z[1/p_1, 1/p_2] = \Z[1/p_1, 1/p_2]$, the result follows by constructing, for all $n <\omega$, an $m_n\in M_\mu$ which takes care of the coordinates of $(a_\alpha)_{\alpha < \mu}$ which are not in $p_1^n \Z[1/p_2]$.

    The \emph{in particular} follows from $a_\alpha \in p_1^\omega \Z[1/p_2]=0$ for all $\alpha \in \mu \backslash \bigcup_{n<\omega} I_n$.
    
    \item It is clear that $p_2^\omega N_\mu = N_\mu$, and $p_1^\omega V_\mu = V_\mu$ follows from Fact \ref{Facts:Uncited}. For $p_2^\omega V_\mu = V_\mu$, the inclusion $p_2^\omega(p_1^\omega N_\mu) \subseteq p_1^\omega N_\mu$ is obvious, and we are left to show the other inclusion. Let $g \in  p_1^\omega N_\mu$ and $k <\omega$. Since $p_2^\omega N_\mu = N_\mu$ there is $h \in N_\mu$ such that $p_2^k h = g$. As $g \in p_1^\omega N_\mu$ and $\gcd(p_1, p_2)=1$, it follows that $h \in p_1^\omega N_\mu$. Hence $g \in p_2^k(p_1^\omega N_\mu)$, and $g \in p_2^\omega(p_1^\omega N_\mu)$ follows.
    \item It is clear that $V_\mu, N_\mu \in K_2(2^\mu)$ as $\| N_\mu \| = 2^\mu$, and $V_\mu \leq_{\Kk} N_\mu$ follows from (2). \qedhere
\end{enumerate}

\end{proof}

\begin{remark}
$\| N_\mu \| = 2^\mu$ is the main reason we have to work with $\Kk(2^\mu)$ instead of $\Kk( \mu)$. Note also that $\| V_\mu \| = \mu^{\aleph_0}$.

    \end{remark}

The reason tameness fails in our example is due to the following algebraic result. The statement is similar to Theorem \ref{Prop:Homo_Unique}, but the proof uses more advanced machinery.

\begin{theorem}\label{Prop:measurable_homo} Assume $\mu$ is less than the first measurable cardinal. If $f:N_\mu \to N_\mu$ is a group homomorphism with $f(e_\alpha) = e_\alpha$ for all $\alpha < \mu$, then  $f = \id_{N_\mu}$.
\end{theorem}
\begin{proof}
   Let $g = f - \id_{N_\mu}$. Note that $g:N_\mu \to N_\mu$ is a group homomorphism with $g(e_\alpha) = 0$ for all $\alpha < \mu$.

    Let $W_\mu = \prod_{\alpha < \mu} \Z[1/p_2] e_\alpha \leq U_\mu$ and observe that $ N_\mu = \left\langle M_\mu, W_\mu \right\rangle \leq U_\mu$. Let $\pi_\beta: N_\mu \to \Z[1/p_1, 1/p_2]$ denote the projection map onto the $e_\beta$-coordinate. We will show $(\pi_\beta \circ g) \restr W_\mu = 0$ for all $\beta < \mu$, and hence that $g \restr W_\mu = 0$. Since $g \restr M_\mu = 0$ and $ N_\mu = \left\langle M_\mu, W_\mu \right\rangle$ this will imply that $g=0$ and thus $f = \id_{N_\mu}$.

    Let $\beta < \mu$ and observe that $(\pi_\beta \circ g) \restr W_\mu : W_\mu = \prod_{\alpha < \mu} \Z[1/p_2] e_\alpha \to \Z[1/p_1, 1/p_2]$ is such that $(\pi_\beta \circ g)(e_\alpha) = 0$ for every $\alpha <\mu$. Hence $(\pi_\beta \circ g) \restr \bigoplus_{\alpha < \mu} \Z[1/p_2] e_\alpha = 0 $

    Since $\Z[1/p_1, 1/p_2]$ is slender
    by Nunke's Criterion (see for example \cite[Corollary IX.2.4]{EkMe15}) and $\mu$ is below the first measurable cardinal, it follows that $(\pi_\beta \circ g) \restr W_\mu = 0 
    $ by the Łoś-Eda Theorem (see for example, \cite[Corollary~III.3.4]{EkMe15}). Hence $\pi_\beta \circ g = 0$.\end{proof}

We have the following complementary result to Theorem \ref{Prop:measurable_homo} which explains why our arguments can not be pushed beyond the first measurable cardinal. See \cite[Theorem 3.22]{mac} for a similar construction.

\begin{lemma} Assume $\mu$ is greater or equal to the first measurable cardinal. There exists a homomorphism $f:N_\mu \to N_\mu$ such that $f \neq \id_{N_\mu}$ and $f(e_\alpha) = e_\alpha$ for all $\alpha < \mu$.
\end{lemma}
\begin{proof}
  Since $\mu$ is greater or equal to the first measurable cardinal, there exists a non-principal countably complete ultrafilter $\mathcal U$ on $\mu$.

  For every $a=(a_\alpha)_{\alpha<\mu}\in N_\mu$ and $q\in \Z[1/p_1, 1/p_2]$, let $I_{a,q}=\{\alpha <\mu: a_\alpha = q\}$. Note that $\mu = \bigcup_{q\in \Z[1/p_1, 1/p_2]} I_{a,q}$ defines a partition of $\mu$ into countably many parts. As $\mathcal U$ is countably complete, for any fixed $a\in N_\mu$, there exists a unique $q\in \Z[1/p_1, 1/p_2]$ with $I_{a,q}\in \mathcal U$. Let $\operatorname{mac}: N_\mu \to \Z[1/p_1, 1/p_2]$ be given by $\operatorname{mac}(a) := q$ where $q$ is the unique element in $\Z[1/p_1, 1/p_2]$ such that $I_{a,q}\in \mathcal U$, i.e., $\operatorname{mac}(a)$ denotes the \emph{most abundant coefficient} of $a\in N_\mu$ with respect to~$\mathcal U$. Observe that $\operatorname{mac}$ is a group homomorphism such that $\operatorname{mac}((0)_{\alpha<\mu})=0$ and $\operatorname{mac}((1)_{\alpha<\mu})=1$.
  
  Let $f:N_\mu \to N_\mu$ be given by $f(a) = a + \operatorname{mac}(a)e_0$. Clearly, $f$ is a group homomorphism. Observe that $f(e_\alpha) = e_\alpha + 0e_0 = e_\alpha$ for all $\alpha <\mu$, but $f((1)_{\alpha<\mu}) = (1)_{\alpha<\mu} + 1e_0 \neq (1)_{\alpha<\mu}$.
\end{proof}

We now prove one of the key results of the paper.

\begin{theorem}\label{Cor:Kk_not_kappa_tame} Assume $\mu$ is a regular  uncountable cardinal less than the first measurable cardinal. Then $\Kk(2^\mu)$ is not $(<\mu)$-tame.
  
\end{theorem}
\begin{proof}
    Let $G= V_\mu$ and $H=  N_\mu$ where $N_\mu, V_\mu$ are as in Definition \ref{Def:K2_M_N}. Let $a=(1)_{\alpha <\mu}, b = (1)_{\alpha <\mu}+e_0 \in H$. Observe  $G \leq_{\Kk} H$ in $\Kk(2^\mu)$  by Proposition \ref{Prop:Basic-Nmu}.

    We show that $\gtp_{\Kk}(a/G; H), \gtp_{\Kk}(b/G; H)$ witness that $\Kk$ is not $(<\mu)$-tame.

    That $\gtp_{\Kk}(a/G; H) \neq \gtp_{\Kk}(b/G; H)$ can be shown exactly as in Claim 1 of Theorem \ref{Thm:Kl_not_aleph0_tame} since $a, b \notin G$ by Proposition \ref{Prop:Basic-Nmu}, the \emph{in particular} of Proposition \ref{Prop:Kk_closure}, $p_2^\omega H = H$ by Proposition \ref{Prop:Basic-Nmu}, and Theorem \ref{Prop:measurable_homo}. We show that for every $C \subseteq G$ with $|C| < \mu$ we have that $\gtp_{\Kk}(a/G; H) \restr C = \gtp_{\Kk}(b/G; H) \restr C$.

    Let $C \subseteq  G = p_1^\omega N_\mu$ with $|C| < \mu$. For every $c =(c_\alpha)_{\alpha<\mu} \in C$, there is $\beta_c < \mu$ such that $c_\alpha = 0$ for every $\alpha \geq \beta_c$ by Proposition \ref{Prop:Basic-Nmu}(1) and the fact that $\operatorname{cf}(\mu)= \mu > \aleph_0$. Let $\beta = \operatorname{sup}_{c\in C} \beta_c$, and observe $\beta <\mu$ because $|C| < \mu$ and $\mu$ is regular. Hence for every $\alpha \geq \beta$ and $c \in C$, $c_\alpha =0$.

    Let $f_a: H \to H$ be given by $f_a((x_\alpha)_{\alpha <\mu})= (x_\alpha)_{\alpha <\mu} + x_\beta e_0$. Observe $f_a$ is an isomorphism such that $f_a(a)=b$ and $f_a\restr C= \id_C$ by the choice of $\beta$. Hence  $\gtp_{\Kk}(a/G; H) \restr C = \gtp_{\Kk}(b/G; H) \restr C$.
\end{proof}

\begin{remark}\label{Remark:up}
In the proof of Theorem \ref{Cor:Kk_not_kappa_tame}, we have that $\|G\| = \mu^{\aleph_0}$, and $\|G\| < 2^\mu =LS(\Kk(2^\mu))$ in most cases. This can be fixed if instead one takes $G':= V_\mu \oplus \bigoplus_{\beta< 2^\mu} \Z e'_\beta$ and $H':=  N_\mu \oplus \bigoplus_{\beta< 2^\mu} \Z e'_\beta$ where $N_\mu, V_\mu$ are as in Definition~\ref{Def:K2_M_N}, and $a=((1)_{\alpha <\mu}, (0)_{\beta <2^\mu}), b = a+e_0 \in H'$. This can be achieved as adding copies of $\Z$ does not change the algebraic properties of $V_\mu, N_\mu$ that we use in the proof of Theorem \ref{Cor:Kk_not_kappa_tame}.
 \end{remark}

 Observe that adding copies of $\Z$ to $V_\mu, N_\mu$, as in Remark \ref{Remark:up}, one can also strengthen Theorem \ref{Cor:Kk_not_kappa_tame}.
 
 \begin{lemma}\label{Lem:Kk_not_mu_theta_tame} Assume $\mu$ is a regular  uncountable cardinal less than the first measurable cardinal.
     For every $\theta \geq 2^\mu$, $\Kk(2^\mu)$ is not $(<\mu, \theta)$-tame.
 \end{lemma}

The family of AECs determined by $\Kk$ can also be used to obtain an example of failure of $(<\aleph_0)$-tameness.

\begin{lemma}\label{kk-no-aleph-0-tame}
    $\Kk(\aleph_0)$ is not $(<\aleph_0)$-tame. Moreover, for every infinite cardinal~$\theta$, $\Kk(\aleph_0)$ is not $(<\aleph_0, \theta)$-tame.
\end{lemma}
\begin{proof}
Let $M, N$ be as in Definition \ref{def:M,N}. As $p_1^\omega M = p_2^\omega M$, it follows from Proposition \ref{Prop:div-MN} that $M, N \in K_2(\aleph_0)$ and $M\leq_{\Kk} N$. Then the types provided in the proof of Theorem \ref{Thm:Kl_not_aleph0_tame} can be used to show that $\Kk(\aleph_0)$ is not $(<\aleph_0)$-tame.

For the \emph{moreover} add copies of $\Z$ as in Remark \ref{Remark:up}.
\end{proof}

\begin{remark}
    A key difference between $\Kl$ and $\Kk(\aleph_0)$ is that for every uncountable cardinal $\theta$, $\Kl$ is $(<\aleph_0, \theta)$-tame while $\Kk(\aleph_0)$ is not $(<\aleph_0, \theta)$-tame. 
\end{remark}

\begin{cor}  Assume  $\kappa$ is greater or equal to the first measurable cardinal. Then $\Kk(p_1, p_2, \kappa)$ is not $(<\mu)$-tame for any regular cardinal $\mu$ less than the first measurable cardinal.
\end{cor}

Based on Section \ref{sectio-3} one would expect that $\Kk(2^\mu)$ would be at least $2^\mu$-tame but we do not know if this is the case.

\begin{question}
    Is $\Kk(2^\mu)$  $2^\mu$-tame? Is it $\mu$-tame?
\end{question}

The following notion is inspired by the notion of amalgamation base \cite{shvi}, \cite{van06}.

\begin{defin}
  Let $\K$ be an abstract elementary class and $\lambda$ be an infinite cardinal. $M \in \K$ is a \emph{$(<\lambda)$-tame base} if for every $p, q \in \gS_\K(M)$, if $p \neq q$, then there is $A \subseteq |M|$ such that $|A|< \lambda$ and $p \restr A \neq q \restr A$. 
\end{defin}

\begin{remark}
    If $M \in \K$ and $\| M \| < \lambda$, then $M$ is $(<\lambda)$-tame base. Hence the notion of $(<\lambda)$-tame base is only interesting when one is working with models of cardinality greater than or equal to $\lambda$.
\end{remark}

To show the abundance of tame bases in $\Kk=\Kk(\kappa)$, we will use the following result which can be shown as in Proposition \ref{Prop:Kl_gtp_equal_to_TF} using Proposition \ref{Prop:Kk_closure}.

\begin{lemma}\label{Prop:Kk_gtp_equal_to_TF} 
Let $\kappa$ be an infinite cardinal.
    If $G \leq_{\Kk} H_1, H_2$ in $\Kk(\kappa)$, $a_1\in H_1$, $a_2 \in H_2$, and $p_2^\omega H_1, p_2^\omega H_2 \subseteq G$, then $\gtp_{\K_{TF}}(a_1/G; H_1) = \gtp_{\K_{TF}}(a_2/G; H_2)$ if and only if $\gtp_{\Kk(\kappa)}(a_1/G; H_1) = \gtp_{\Kk(\kappa)}(a_2/G; H_2)$.
\end{lemma}

\begin{lemma} Let $\kappa$ be an infinite cardinal.
If $0\ne G \in K_2(\kappa)$, then there is $H \in K_2(\kappa)$ a $(<\aleph_0)$-tame base with $G \leq_{\Kk} H$ and $\| G \| = \| H \|$.
\end{lemma}
\begin{proof}
    If $p_1^\omega G = p_2^\omega G$, let $H = G \oplus \mathbb{Z}[1/p_2]$, otherwise let $H = G$. It is clear that $H \in K_2(\kappa)$, $G \leq_{\Kk} H$, and $\| G \| = \| H \|$. The same argument as that given in Case~2 of Lemma \ref{Lem:Kl_aleph1_tame} using Lemma \ref{Prop:Kk_gtp_equal_to_TF} can be used to show that $H$ is a $(<\aleph_0)$-tame base. 
\end{proof}

We show that $\Kk$ has some weak stability.

\begin{lemma}\label{Thm:Kk_almost_stable} Let $\kappa$ be an infinite cardinal.
   If $\lambda^{\aleph_0} = \lambda$ and $\lambda \geq \kappa$, then $\Kk(\kappa)$ is almost stable in $\lambda$.
\end{lemma}
\begin{proof}
    Suppose $\lambda^{\aleph_0} = \lambda$ and let $G, H \in K_2(\kappa)$ with $\|G\|=\lambda$ and $G \leq_{\Kk} H$. Assume for sake of contradiction that $|\gS_{\Kk}(G; H)| > \lambda$. Let $\{\gtp_{\Kk}(a_i/G; H) : i<\lambda^+\}$ be an enumeration of distinct types in $\gS_{\Kk}(G; H)$. Let $G' = \Cl_{\Kk}^H(G \cup p_2^\omega H) = \Cl_{\K_{TF}}^H(G \cup p_2^\omega H)$ and $\Phi: \lambda^+ \to \gS_{\K_{TF}}(G')$ be given by $\Phi(i)=\gtp_{\K_{TF}}(a_i/G'; H)$. 

    Since $\| G' \| = \lambda$ (because $|p_2^\omega H|\le \kappa \le \lambda$) and $\K_{TF}$ is stable in $\lambda$ by Fact \ref{Thm:Properties_TF}, there exist $i\neq j <\lambda^+$ such that $\gtp_{\K_{TF}}(a_i/G'; H) = \gtp_{\K_{TF}}(a_j/G'; H)$. Since $G' \leq_{\Kk} H$, it follows from Lemma \ref{Prop:Kk_gtp_equal_to_TF} that $\gtp_{\Kk}(a_i/G'; H) = \gtp_{\Kk}(a_j/G'; H)$. As $G \leq_{\Kk} G'$, we have that $\gtp_{\Kk}(a_i/G; H) = \gtp_{\Kk}(a_j/G; H)$, a contradiction.
\end{proof}

We conclude with the proof of Theorem \ref{Thm:main_Kk}.

\begin{proof}[Proof of Theorem \ref{Thm:main_Kk}] It follows from Remark \ref{Remark:Kk_AEC_Similar_Kl} that 
    $\Kk = \Kk(p_1, p_2, \kappa)$ is an AEC with $\LS(\Kk) = \kappa$ and that (3) and (4) hold.
    \begin{enumerate}
        \item Follows from Theorem \ref{Cor:Kk_not_kappa_tame}.
        \item Follows from Lemma \ref{Thm:Kk_almost_stable}.\qedhere
    \end{enumerate}
\end{proof}

\end{document}